\documentclass[12pt]{article}

\usepackage{amssymb,amsmath}
\usepackage{enumerate, xspace}
\usepackage{dsfont}
\usepackage{changebar}
\usepackage[colorlinks]{hyperref}
\usepackage[dvips]{graphicx}
\usepackage{psfrag}
\usepackage{amsthm}

\numberwithin{equation}{section}
\newtheorem{thm}{Theorem}[section]
\newtheorem{lem}[thm]{Lemma}

\newtheorem{conj}[thm]{Conjecture}

\theoremstyle{definition}

\newcommand\cA{{\mathcal A}}
\newcommand\cB{{\mathcal B}}
\newcommand\cC{{\mathcal C}}
\newcommand\cD{{\mathcal D}}

\newcommand{\be}{\begin{equation}}
\newcommand{\ee}{\end{equation}}
\newcommand{\bs}{\begin{split}}
\newcommand{\es}{\end{split}}

\begin{document}
\title{A bijectional attack on the Razumov-Stroganov conjecture}
\date{\today}
\author{Arvind Ayyer\\
\small Department of Physics\\
\small 136 Frelinghuysen Rd\\
\small Piscataway, NJ 08854.\\
\small \texttt{ayyer@physics.rutgers.edu}\\
\\
Doron Zeilberger\\
\small Department of Mathematics\\
\small 110 Frelinghuysen Rd\\
\small Piscataway, NJ 08854.\\
\small \texttt{zeilberg@math.rutgers.edu}\\
}

\maketitle

\begin{abstract}
We attempt to prove the Razumov-Stroganov conjecture using a
bijectional approach. We have been unsuccessful but we believe the
techniques we present can be used to prove the conjecture.
\end{abstract}

\section{Introduction}
Ever since the discovery of alternating sign matrices (ASM), the
conjecture on the number of such matrices \cite{mrr} and the proof of
this conjecture in \cite{z1} (followed by a shorter proof
in \cite{k}), a number of interesting structures
have been found which are counted by the ASM numbers \cite{p}. One of
these structures is that of fully packed loops which has given rise to
another beautiful conjecture, formulated in \cite{ngb,rs}, and now popularly
known as the Razumov-Stroganov conjecture. 

This conjecture, motivated by some exactly solvable models in
statistical physics, is about seven years old and already several
connections have been found to other combinatorial objects. The
literature on this subject is already very large and we will not
be able to survey the complete literature. Readers are referred to the
papers \cite{deg,df1,zj} for more references.

The proof has been found only in some restricted cases \cite{zj}. This
has also motivated a bunch of refined conjectures, starting with
\cite{df1}. All of these proofs and refined conjectures generally
involve connections to other combinatorial objects and might be a red
herring. In this short letter, we suggest a completely self-contained
approach which is more combinatorial. This might lead to a proof of the
Razumov-Stroganov conjecture.\footnote{Yet another conjecture on a
  subject where all papers seem to be conjectures.}

The paper is organized as follows. In Section~\ref{sec:defn}, we give
the basic definitions and set the notation. In Section~\ref{sec:comb},
we restate the conjecture combinatorially, and in
Section~\ref{sec:altpath}, we propose a mechanism for tackling this
restatement. Lastly, in Section~\ref{sec:exp}, we mention some
computer experiments which have been successful for smaller sizes but
have not yielded the proposed bijection.

\section{Definitions} \label{sec:defn}
Fully Packed Loops (FPLs) of size $n$ are configurations of lattice
paths which are drawn in a square lattice of dimensions $n \times n$
as follows.  One assigns the labels $1,\dots,2n$ on the $4n$ endpoints of
the lattice alternatively. Subsequently, one connects the endpoint
labels pairwise via paths on the lattice so that there are no
crossings. Note that one is allowed to fill in
loops also. The important point is that these lattice paths have to
cover the entire square lattice. It turns out that FPLs are in bijection
with  Alternating Sign Matrices (ASMs) and therefore the number of
FPLs is given by the ASM numbers $A_n$ \cite{z1,k},
\be
A_n = \prod_{i=0}^{n-1} \frac{(3i+1)!}{(n+i)!}.
\ee
Let us denote the set of FPLs of size $n$ as $F_n$.
The number of ways of connecting $2n$ endpoints on a circle is the
Catalan number $C_n$. We can count the number of FPLs according to
the connectivity of its endpoints. Let the
connectivities be labelled $\pi_j$, $j=1,\dots,C_n$. Then we let
$A(\pi_j)$ be the number of FPLs with connectivity $\pi_j$.

On the set of endpoints, one can define the operators $e_i$ for $i \in
[1,2n]$ in the following way. Suppose in a particular connectivity $\pi$,
$i$ joins $j$ and $i+1$ joins $k$. Then $e_i(\pi)$ is the new
connectivity obtained by connecting $i$ to $i+1$ (cyclically) 
and $j$ to $k$. It is
easy to see that this is a valid connectivity also.
These $e_i$ satisfy the defining relations of the Temperley-Lieb
algebra,
\be
\begin{split}
e_i^2 &= e_i, \\
[e_i,e_j] &= 0, \qquad \text{for } |i-j| \geq 2, \\
e_i e_{i \pm 1} e_i &= e_i.
\end{split}
\ee

Consider the vector space whose basis consists of all
connectivities $\pi_j$, $j=1,\dots,C_n$. Since the operators $e_i$,
$i=1,\dots,2n$ act 
on connectivities, one can 
construct a matrix for each $e_i$ in that basis. Since every 
$e_i$ takes any basis vector to another single basis vector, the
matrix must have a single one in every column, the rest being
zeros. This ensures that the vector $w=(1,\dots,1)$ satisfies $w e_i = w$ for
all $i$.
Since $e_i$ is a nonnegative integer matrix, it has a largest
eigenvalue by the Perron-Frobenius theorem, the corresponding
eigenvector being positive. Thus, $w$ is the left Perron-Frobenius
eigenvector for each $e_i$.

Define the
``Hamiltonian'' matrix,
\be
H = \sum_{i=1}^{2n} e_i.
\ee
$w$ is also the left eigenvector of $H$ with the largest eigenvalue
$2n$. Obviously, $H$ also has a single right eigenvector with the same
eigenvalue. Call it $\Psi$. We can choose $\Psi$ to have positive
integer entries, and can expand it in our original basis. Then the
Razumov-Stroganov conjecture states that 
\be \label{conjrs}
\Psi = \sum_{j=1}^{C_n} A(\pi_j) \pi_j.
\ee

%Let $f$ be an FPL. Then we set $\pi(f)$ to be the connectivity of
%$f$. We want to define a map $\hat{e}_i$ on FPLs so that
%$\pi(\hat{e}_i(f)) = e_i(\pi(f))$. If we can do this in a precise
%algorithmic way, we are well on our way. 

\section{A Combinatorial Restatement of the Conjecture} \label{sec:comb}
To restate the conjecture we rewrite $H\Psi$,
\be
\begin{split}
\left( \sum_{i=1}^{2n} e_i \right) \left( \sum_{j=1}^{C_n} A(\pi_j)
\pi_j \right) 
&= \sum_{j=1}^{C_n} A(\pi_j) \left( \sum_{i=1}^{2n} e_i \pi_j \right), \\
&= \sum_{j=1}^{C_n} A(\pi_j) \left( \sum_{\substack{i=1,k=1 \\ e_i \pi_j =
    \pi_k}}^{2n,C_n} \pi_k \right), \\
&= \sum_{k=1}^{C_n} \pi_k \left( \sum_{\substack{i=1,j=1 \\e_i
    \pi_j = \pi_k}}^{2n,C_n} A(\pi_j) \right).
\end{split}
\ee
Since $H \Psi = 2n\Psi$ and the $\pi_k$'s are independent vectors, we
must have
\be \label{conjbij}
\sum_{\substack{i=1,j=1 \\e_i
    \pi_j = \pi_k}}^{2n,C_n} A(\pi_j) = 2n A(\pi_k), \qquad \text{for
  all } k=1,\dots,C_n.
\ee
This restatement has the advantage that it involves only the number of
FPLs for a given connectivity and in that sense, is more
combinatorial. There is no mention of any matrices, eigenvalues and
algebras. To the best of our knowledge, this is the first time the
Razumov-Stroganov conjecture has been interpreted this way.

We comment that another way to  think of \eqref{conjbij} is that the
map $\pi \to A(\pi)$ is a ``harmonic function'' on the
graph of connectivities, where $\pi$ has the $2n$ neighbors
$e_i \pi$ ($i=1, \dots, 2n$).

\begin{figure}[ht!]
\hfil \includegraphics[width=3.5cm]{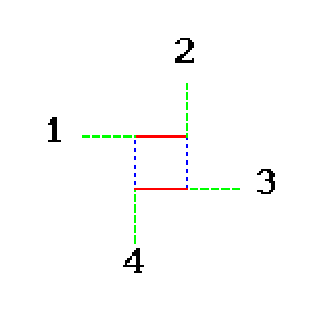} 
\hfil \includegraphics[width=3.5cm]{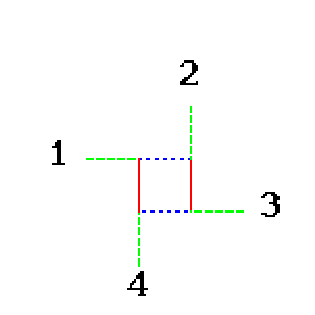}
\caption{The two FPLs for $n=2$. The red lines represent the paths.}
\label{fig:fpl2}
\end{figure}

Let us consider a simple example. 
When $n=2$, we have two
connectivities on four points, 
$\pi_1 = (12)(34), \, \pi_2 = (14)(23)$ with $A(\pi_1) = A(\pi_2) =
1$. These FPLs are shown in Figure~\ref{fig:fpl2}. The operators $e_i$
act in a straightforward manner: 
\be
\begin{split}
e_1(\pi_1) = e_1(\pi_2) = e_3(\pi_1) = e_3(\pi_2) = \pi_1, \\
e_2(\pi_1) = e_2(\pi_2) = e_4(\pi_1) = e_4(\pi_2) = \pi_1.
\end{split}
\ee
For $k=1$, the left hand side is $2A(\pi_1)+2A(\pi_2)$ (when $i=1,3$),
which is also equal to the right hand side, $4A(\pi_1)$. Similarly for $k=2$.

\eqref{conjbij} can be interpreted as a statement of the equality of
the cardinality of two sets in the following way. Let
\be
\begin{split}
\cA_k &= \{ (i,f) | i \in [1,2n], f \in F_n, \pi(f) = \pi_k \}, \\
\cB_k &= \{ (j,g) | j \in [1,2n], g \in F_n, e_j(\pi(g)) = \pi_k \},
\end{split}
\ee
be two subsets of the product set $[1,2n] \times F$. Then the restatement
\eqref{conjbij} is equivalent to the statement
\be \label{setbij}
|\cA_k | = |\cB_k| \qquad \text{for all } k = 1,\dots,C_n.
\ee

\section{Alternating Paths} \label{sec:altpath}
The equation \eqref{setbij} is a conjectured equinumeracy where the index
varies among all link patterns. We conjecture that something stronger
is true. Namely, a bijection in which the index varies among all FPLs.
We suggest that is possible, given an FPL
$f$ of size $n$ and an integer $i$ between $1$ and $2n$,  to choose in
an invertible way another FPL $g$ and another 
integer $j$ between $1$ and $2n$ so that $e_i(\pi(g)) = \pi(f)$. The
role of $j$ would be to provide the inversion. We
give in this section such a procedure, which if used correctly, would
yield precisely the correct bijection. The main idea is that of
{\em alternating paths}.

Recall that the FPL of size $n$ can be represented as a subset of the
edges of the two dimensional square lattice whose vertex coordinates
lie between $1$ and $n$ plus some additional edges needed to describe
the alternating boundaries. For the purposes of defining alternating
paths we do not need the additional boundary edges. Let us define the
set of all possible edges: 
$$S = \{ [(i,j),(i+1,j)],[(i,j),(i,j+1)] | \, i,j \in (1,n-1) \},$$
which has cardinality $2n(n-1)$. The FPL is then simply defined by a
set of edges 
$E \subset S$ which has cardinality $n(n-1)$. Then 
one can define the converse FPL by the set $\bar E = S \setminus
E$ and of course $\bar{\bar E} = E$. In Figure~\ref{fig:fpl2}, the
blue lines represent the interior part of the converse FPL.

We define an alternating path by a closed loop $L=[e_1,e_2,\dots,e_m]$
in an FPL in which each edge $e_i$ belongs to the converse of the set
that $e_{i-1}$ belongs to. 
That is each edge is alternately in either $E$ or $\bar E$. This
notion of alternating paths leads immediately to an easy lemma.

\begin{lem} \label{lem:altpath}
Given an FPL of size $n$ by its edges $E$ and given an alternating
path $L$ in the FPL, create a new edge set $E'$ by the following
procedure. For every edge $e$,
\begin{itemize}
\item if $e$ belongs to $E$ and does not belong to $L$, set $e \in E'$;
\item if $e$ belongs to $L$ and does not belong to $E$, set $e \in E'$;
\item else do nothing.
\end{itemize}
This procedure gives a new valid FPL.
\end{lem}

\begin{proof}
The idea is very simple. If we color the edge sets $E$ red and $\bar
E$ blue, then what this procedure does is to simply interchange the
colors within the alternating path. The defining property of an FPL is
that every vertex will have two red and two blue edges connected to it
and this procedure does not change that. For the edges at the
boundary, we either have vertices with one red and one blue edge at
the corners or those with one red and two blue edges or those with two
red and one blue edges. At each of these places at most one red and one
blue edge are rearranged and thus the procedure preserves the
arrangement and is therefore a valid FPL.
\end{proof}

\begin{figure}[ht!]
\includegraphics[width=3.5cm]{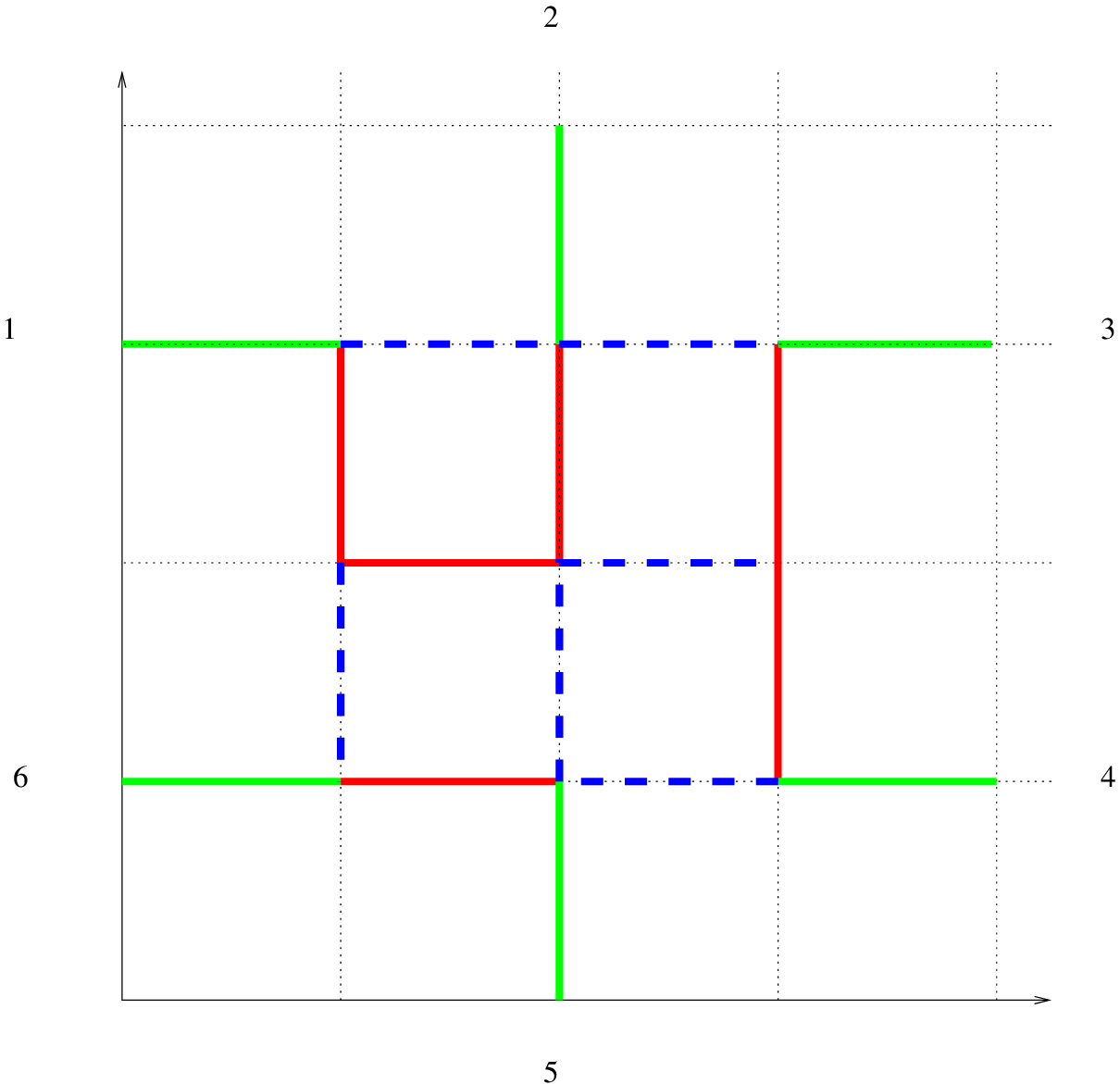}  \hfil
\includegraphics[width=3.5cm]{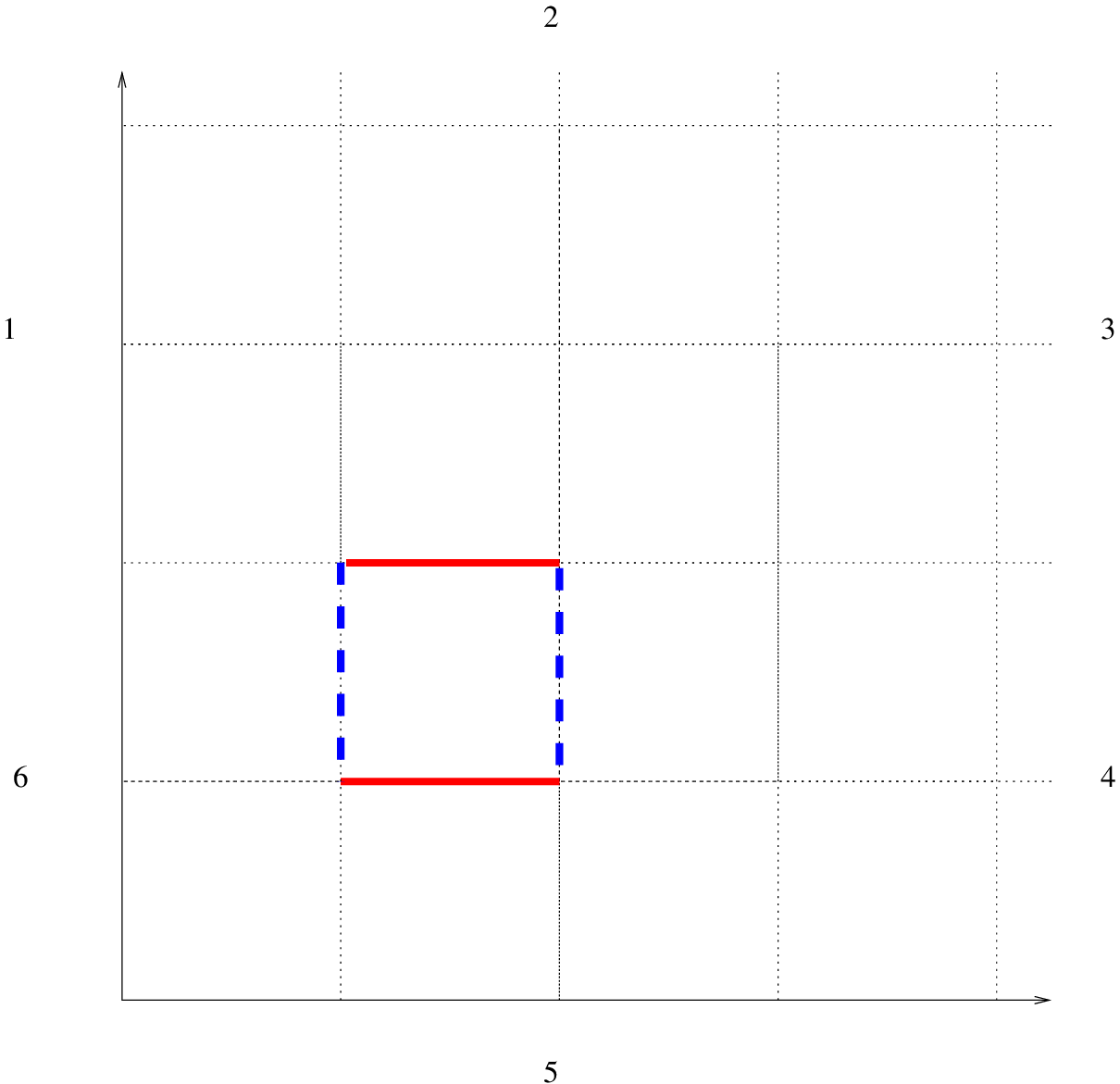} \hfil
\includegraphics[width=3.5cm]{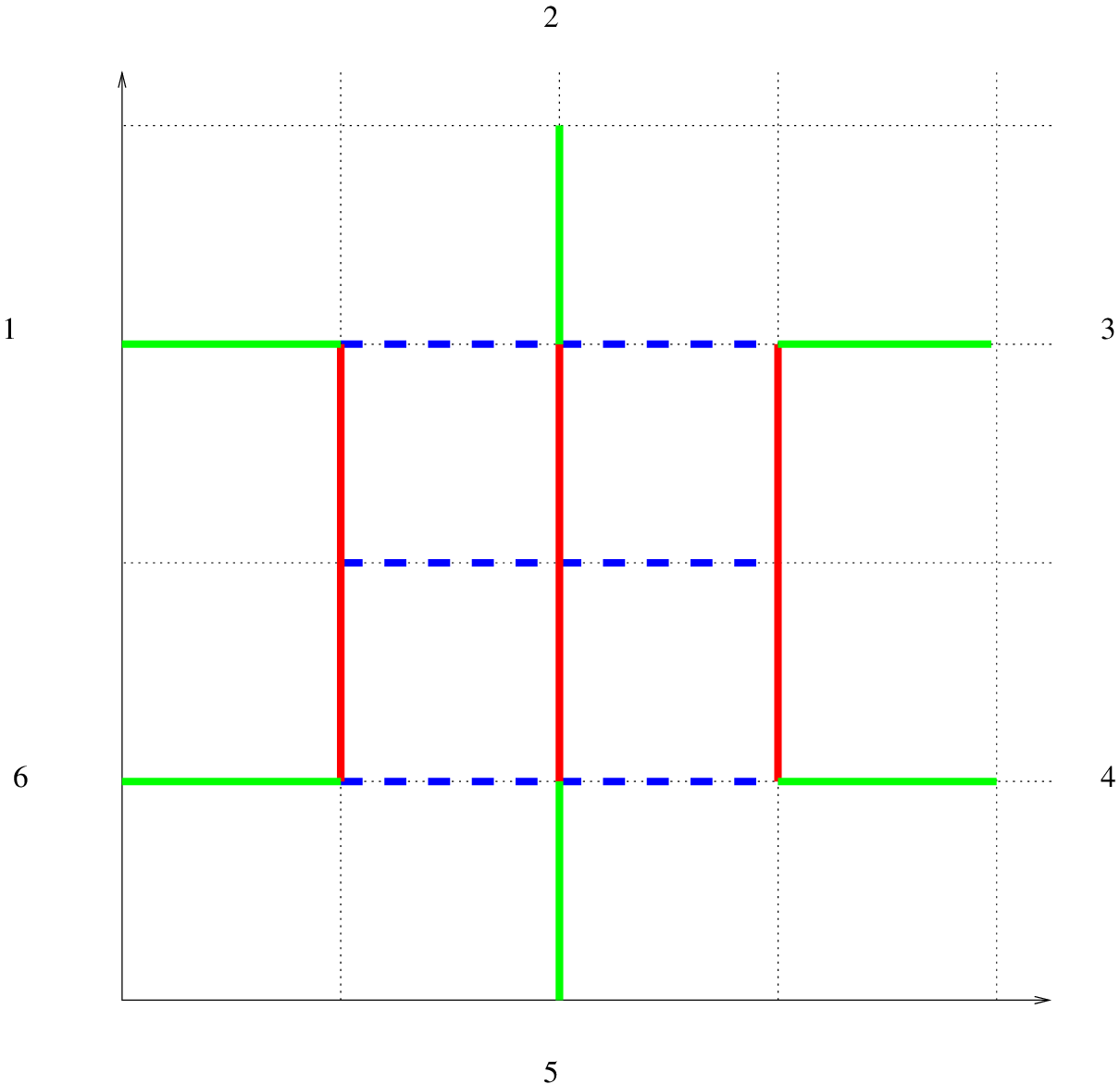}  
\caption{An example of an alternating path. Suppose we start with the
  FPL on the left and consider the alternating 
  path of the figure on the right, we end up, after the operation of
  Lemma~\ref{lem:altpath}, with the FPL on the right.}
\label{fig:fpleg}
\end{figure}

We now make the conjecture made at the beginning of this section
more precise.
\begin{conj} \label{conj:bij}
Given an FPL $f$ of size $n$ and an integer $i$ between $1$ and $2n$
there is a canonical algorithm to find an alternating path $L$, which
leads to another FPL $g$ in 
which the paths from $i$ and $i+1$ are connected. This algorithm also
leads to an integer
$j$ which will be needed to find the inverse of this map.
\end{conj}

\section{Experiments} \label{sec:exp}
We have performed computer experiments using Maple to find the
supposed bijection Conjecture~\ref{conj:bij} and present a set of
programs in two packages titled {\em RS} and {\em FPL}. These packages
are available from the homepages of the authors and the \texttt{arXiv}.

The simplest test of the above conjecture is to simply count the
number of times each FPL is the output of the alternating path
procedure. If each FPL occurs exactly $2n$ times, this is a hint that
the algorithm is correct.
Most of the algorithms work for $n=2$ and $n=3$. The first nontrivial
test occurs for $n=4$ where there are 42 FPLs. This case is also
original in the sense that there occurs an FPL with a loop
inside. This forces the algorithm to be nontrivial.

The first comment is that the obvious algorithms for the conjectured
alternating paths do not work. Neither choosing the first available
alternating path (starting from the path beginning at $i$) nor
choosing the smallest alternating path which 
connects $i$ to $i+1$ are well-defined operations. There do arise
examples when $n \geq 5$ in which there are several paths beginning at
the first edge, two of which have the same length. Similarly there are
examples where there are multiple shortest alternating paths which
lead to different final FPLs. 
This problem of
choice in proving the Razumov-Stroganov conjecture has been noticed
before \cite{df2}.

The second comment is that one can use the dihedral symmetry inherent
in the FPL picture proved in \cite{w}. If one has a prescribed
alternating path
algorithm for a certain FPL $f$ and an integer $i$ $(f,i) \to (g,j)$,
then one can use 
the rotation defined in \cite{w} to define the algorithm for
$(f',i+1)$ as $(g',j+1)$. And similarly for the reflected case.

Another possible way out is to introduce ``catalytic'' sets.  It is
very possible that two sets $\cA$ and $\cB$ have the same cardinality
without being in ``natural bijection'', but there exist other sets
$\cC$ and $\cD$, such that it is known that $|\cC|=|\cD|$ and there is
a natural bijection between the Cartesian products $\cA \times \cC$
and $\cB \times \cD$. Similarly, if $\cA \cup \cC$ has a natural
bijection with $\cB \cup \cD$. In other words, one has not to be a
fanatic about ``pure bijective proofs'', but view them merely as yet
another tool that may be combined with other tools of the trade.

\section*{Acknowledgements}
We thank Philippe Di Francesco for telling us about the
  connection to Reference \cite{df2}.


\begin{thebibliography}{99}
\bibitem{mrr} W.H. Mills, D.P. Robbins and H. Rumsey Jr., Alternating
  Sign Matrices and Descending Plane Partitions, {\em
    J. Combin. Th. Ser. A} {\bf 34} (1983), 340--359. 
\bibitem{z1} D. Zeilberger, Proof of the Alternating Sign Matrix
  Conjecture, {\em Electronic J. Combin} {\bf 3} no. 2 (1996), R13, 84pp.
\bibitem{k} G. Kuperberg, Another Proof of the Alternating-Sign Matrix
  Conjecture", {\em Internat. Math. Res. Notes} No. 3 (1996),
  139--150. 
\bibitem{p} J. Propp, The many faces of alternating-sign matrices,
  {\bf preprint}, arXiv:math/0208125.
\bibitem{ngb} B. Nienhuis, J. de Gier and M.T. Batchelor, The quantum
  symmetric XXZ chain at Delta=-1/2, alternating sign matrices and
  plane partitions, {\em J. Phys. A} {\bf 34} (2001), L265--L270.
\bibitem{rs} A.V. Razumov and Yu.G. Stroganov, Spin chains and
  combinatorics, {\em J. Phys. A} {\bf 34} (2001), 3185--3190. \\
A.V. Razumov and Yu.G. Stroganov, Combinatorial nature of
  ground state vector of O(1) loop model, {\em Theor. Math. Phys.} {\bf
    138} (2004), 333--337, \\
A.V. Razumov and Yu.G. Stroganov, O(1) loop model with different
boundary conditions and symmetry classes of alternating-sign matrices,
{\em Theor. Math. Phys.} {\bf 142} (2005), 237--243.
%\bibitem{d} P. Duchon, On the link pattern distribution of quarter-turn symmetric FPL configurations, {\bf preprint},  arXiv:math/0711:2871.
%\bibitem{ckln} Fabrizio Caselli, Christian Krattenthaler, Bodo Lass
%and Philippe Nadeau, On the number of fully packed loop
%configurations with a fixed associated matching, {\em Electronic
%J. Combin.} {\bf 11} no. 2 (2005), R16, 43 pp.
\bibitem{deg} J. de Gier, Loops, matchings and alternating-sign matrices,
{\em Discr. Math.} {\bf 298} (2005), 365--388.
\bibitem{df1} P. Di Francesco, A refined Razumov-Stroganov conjecture,
  {\em J. Stat. Mech.}, P08009 (2004), \\
P. Di Francesco, A refined Razumov-Stroganov conjecture II,
  {\em J. Stat. Mech.}, P11004 (2004), 
\bibitem{zj} P. Zinn-Justin, Proof of Razumov-Stroganov conjecture for
  some infinite families of link patterns, {\bf preprint}, arXiv:math/0607183
\bibitem{df2} P. Di Francesco, Totally Symmetric Self-Complementary
  Plane Partitions and Quantum Knizhnik-Zamolodchikov equation: a
  conjecture, {\em J. Stat. Mech.}, P09008 (2006).
\bibitem{w} B. Wieland, Large Dihedral Symmetry of the Set of
  Alternating Sign Matrices, {\em Electronic
    J. Combin.} {\bf 7} (2000), R37, 13 pp.


\end{thebibliography}
\end{document}